\documentclass[a4paper]{scrartcl}
\usepackage{amssymb, amsmath, amsthm, mathrsfs, pxfonts, braket}

\usepackage[T1]{fontenc}  
\usepackage{hyperref}     

\usepackage[all,2cell]{xy}
\objectmargin+{1mm}
\labelmargin+{0.8mm}
\SelectTips{cm}{12}

\usepackage{lmodern}

\usepackage{enumitem}
\setlist[enumerate]{itemsep=0pt,label=$(\mathrm{\roman*})$, topsep=5pt}
\setlist[description]{itemsep=0pt}
\setlist[itemize]{itemsep=0pt, topsep=5pt, 
labelindent=\parindent,leftmargin=*}

\newtheorem{thm}{Theorem}[section]

\newtheorem{lem}[thm]{Lemma}

\theoremstyle{definition}
\newtheorem{definition}[thm]{Definition}

\newtheorem{claim}{Claim}

\usepackage{etoolbox}
\usepackage{framed}
\newcommand\enclosebox[2]{%
  \BeforeBeginEnvironment{#1}{\begin{#2}}%
  \AfterEndEnvironment{#1}{\end{#2}}%
}

\enclosebox{thm}{leftbar}
\enclosebox{cor}{leftbar}
\enclosebox{lem}{leftbar}
\enclosebox{prop}{leftbar}
\enclosebox{definition}{leftbar}
\enclosebox{rem}{leftbar}
\enclosebox{ex}{leftbar}

\newcommand{\ab}{\mathrm{ab}}
\newcommand{\Char}{\operatorname{char}}

\newcommand{\Cu}{\operatorname{Cu}}
\newcommand{\CuX}{\Cu(X)}

\newcommand{\CuXp}{\Cu(X')}

\newcommand{\Cbar}{\overline{C}}
\newcommand{\Cinf}{C_{\infty}}
\newcommand{\Cpbar}{\overline{C'}}
\newcommand{\Cpinf}{C'_{\infty}}

\newcommand{\Ct}{C^{\mathrm{t}} }
\newcommand{\CtX}{\Ct(X)}
\newcommand{\CtXg}{\CtX^{\geo}}

\newcommand{\CXg}{\CX^{\geo}}

\newcommand{\CXp}{C(X')}

\newcommand{\CX}{C(X)}

\newcommand{\Cf}{\textit{cf.}\;}
\newcommand{\Coker}{\operatorname{Coker}}

\renewcommand{\d}{\partial}

\newcommand{\DM}{\operatorname{DM}}

\newcommand{\F}{\mathcal{F}}

\newcommand{\Fx}{F(x)}
\newcommand{\Fxp}{F(x')}
\newcommand{\Fxt}{\Fx^{\times}}
\newcommand{\Fxpt}{\Fxp^{\times}}

\newcommand{\FC}{F(C)}
\newcommand{\FCp}{F(C')}
\newcommand{\FCx}{\FC_x}
\newcommand{\FCpxp}{\FCp_{x'}}

\newcommand{\Fbar}{\ol{F}}

\newcommand{\geo}{\mathrm{geo}}

\newcommand{\Gal}{\operatorname{Gal}}
\newcommand{\Gm}{\mathbb{G}_{m}}

\newcommand{\hNis}{h_0^{\mathrm{Nis}}}

\newcommand{\Hom}{\operatorname{Hom}}

\newcommand{\isomto}{\stackrel{\simeq}{\to}}
\renewcommand{\Im}{\operatorname{Im}}
\newcommand{\inj}{\hookrightarrow}

\newcommand{\IC}{I(C)}
\newcommand{\ICp}{I(C')}

\newcommand{\ICpbar}{I(\Cpbar)}

\newcommand{\IX}{I(X)}
\newcommand{\IXp}{I(X')}

\newcommand{\Jac}{\operatorname{Jac}}

\newcommand{\Ker}{\operatorname{Ker}}

\newcommand{\kC}{k(C)}
\newcommand{\kCx}{k(C)_x}
\newcommand{\kx}{k(x)}
\newcommand{\kxt}{\kx^{\times}}

\newcommand{\kX}{k(X)}

\renewcommand{\L}{\mathbb{L}}

\newcommand{\m}{\mathfrak{m}}
\newcommand{\mK}{\m_K}

\newcommand{\N}{\mathbb{N}}
\newcommand{\Np}{\N'}

\newcommand{\ol}[1]{\overline{#1}}

\newcommand{\OK}{O_K}

\newcommand{\OKt}{\OK^{\times}}

\newcommand{\opcit}{\textit{op.\,cit.}}
\newcommand{\onto}[1]{\stackrel{#1}{\to}}
\newcommand{\otimesZ}{\otimes_{\Z}}
\newcommand{\otimesM}{\stackrel{M}{\otimes}}

\newcommand{\PF}{$(\mathbf{PF})\,$}

\newcommand{\phibar}{\overline{\phi}}

\newcommand{\piab}{\pi_1^{\ab}}
\newcommand{\pitab}{\pi_1^{\mathrm{t},\ab}}
\newcommand{\piabC}{\piab(C)}

\newcommand{\piabX}{\piab(X)}

\newcommand{\piabXg}{\piab(X)^{\geo}}
\newcommand{\pitabX}{\pitab(X)}
\newcommand{\pitabXg}{\pitab(X)^{\geo}}

\newcommand{\plim}{\varprojlim}

\newcommand{\Res}{\operatorname{Res}}
\newcommand{\Red}{$(\mathbf{Red})\,$}

\newcommand{\rhoX}{\rho_X}

\DeclareMathOperator*{\restprod}%
 {\mathchoice{\ooalign{\ensuremath{\displaystyle\prod}\crcr\ensuremath{\displaystyle\coprod}}}%
             {\ooalign{\ensuremath{\textstyle\prod}\crcr\ensuremath{\textstyle\coprod}}}%
             {\ooalign{\ensuremath{\scriptstyle\prod}\crcr\ensuremath{\scriptstyle\coprod}}}%
             {\ooalign{\ensuremath{\scriptscriptstyle\prod}\crcr\ensuremath{\scriptscriptstyle\coprod}}}%
 }
\newcommand{\sX}{\mathscr{X}}
 
\newcommand{\surj}{\twoheadrightarrow}
\newcommand{\ssm}{\smallsetminus}
\newcommand{\Spec}{\operatorname{Spec}}
\newcommand{\SK}{\operatorname{SK}}
\newcommand{\SKX}{\SK_1(X)}

\newcommand{\Split}{$(\mathbf{Split})\,$}

\newcommand{\vp}{\varphi}
\newcommand{\vpbar}{\overline{\varphi}}

\newcommand{\vK}{v_K}

\newcommand{\wt}[1]{\widetilde{#1}}
\newcommand{\wh}[1]{\widehat{#1}}

\newcommand{\X}{\wt{X}}

\newcommand{\Xbar}{\ol{X}}

\newcommand{\Z}{\mathbb{Z}}
\newcommand{\Zhat}{\wh{\Z}}

\newcommand{\sn}{\smallskip\noindent}

\title{Class field theory for products of open curves over a local field}
\author{Toshiro Hiranouchi}

\begin{document}
\pagenumbering{arabic}
\maketitle

\begin{abstract}
We investigate the class field theory for products of open curves over a local field. 
In particular, we determine the kernel of the reciprocity homomorphism.
\end{abstract}

\section{Introduction}
\label{Introduction}

In \cite{Hir10} and \cite{Hir17}, we study 
the class field theory 
for open (=non proper) curves over a local field. 
Here, a \textbf{local field} means a complete discrete valuation field 
with finite residue field. 
The aim of this note is to 
extend the above results to products of those open curves (over a local field). 
%

To state our results precisely, 
we use the following notation: 
\begin{itemize}
	\item 
	$k$\,: a local field with $\Char(k) = p\ge 0$, 
	\item 
	$\Xbar_1,\ldots, \Xbar_n$\,: proper and smooth curves over $k$, 
	\item $X_i\subset \Xbar_i$\,: a nonempty open subscheme in $\Xbar_i$ 
	with $X_i(k)\neq \emptyset$ for each $1 \le i \le n$, and 
	\item $X = X_1\times \cdots \times X_n \subset \Xbar = \Xbar_1 \times \cdots \times \Xbar_n$. 
\end{itemize}
First, we introduce an abelian group  
$\CX$  which is called the \textbf{id\`ele class group} 
for $X$  
as in \cite{Hir10} (Def.~\ref{def:CX}), 
and the \textbf{reciprocity map} 
\[
\rhoX:\CX \to \piabX
\] 
is defined using the $2$-dimensional local class field theory(Def.~\ref{def:rho}). 
Next, we determine the prime to $p$-part of the kernel $\Ker(\rhoX)$ 
of $\rhoX$ under some assumptions. 

\begin{thm}[Thm.~\ref{thm:main}]
\label{thm:main_intro}
Let $X$ and $\Xbar$  
be as above. 
Assume the following conditions 
\Red or \Split for each $\Xbar_i$: 
\begin{description}
	\item[\Red] the Jacobian variety $\Jac(\Xbar_i)$ of $\Xbar_i$ has potentially good 
reduction,\  
	\item[\Split] the special fiber of the connected component of the N\'eron model 
of $\Jac(\Xbar_i)$ is an extension of an abelian variety 
by a split torus. 
\end{description}

\noindent
Then, the kernel $\Ker(\rhoX)$ is the maximal $l$-divisible subgroup of $\CX$ 
for all prime number $l\neq p$.  
\end{thm}

The conditions in Thm.~\ref{thm:main_intro} 
are the same one in \cite{Yam09}, Thm.~1.1. 
In \opcit, for the case where $\Xbar_i = X_i$ 
the kernel of $\rho_X$ is determined as in Thm.~\ref{thm:main_intro}
as in the unramified class field theory.
Our main contribution is 
to show that the proof of \opcit\ can also be applied to open curvesover 
the local field has positive characteristic.

\subsection*{Notation}

%
In this note,  
we basically follow the notation and the definitions 
in \cite{Hir17}. 
We fix the following notation:  
A \textbf{variety} over a field $F$ means 
a separated and connected scheme of finite type over $\Spec(F)$.
For a variety $X$, a closed integral subscheme of $X$ of dimension one 
is called a \textbf{curve in $X$}. 

Following \cite{Hir17}, Def.~3.1, 
we call a pair $X \subset \Xbar$ of 
\begin{itemize}[label=$\circ$]
\item $\Xbar$\,:  a smooth, proper and connected curve over a field $F$,  and 
\item $X$\,: a nonempty open subscheme of $\Xbar$
\end{itemize}
an \textbf{open curve} over $F$.

For a field $F$, 
\begin{itemize}
\item $\Char(F)$\,: the characteristic of $F$, 
\item $\Fbar$\,: a separable closure of $F$,  
\item $G_F:=\Gal(\Fbar/F)$\,: the Galois group of the extension $\Fbar/F$, 
\item $F^{\ab}$\,: the maximal abelian extension of $F$ in $\Fbar$, 
\item $G_F^{\ab}:= \Gal(F^{\ab}/F)$\,: the Galois group of $F^{\ab}/F$, 
and 
\item $K_2(F)$\,: the Milnor $K$-group of $F$ (\Cf \cite{Mil70}; \cite{FV}, Chap.~IX). 
\end{itemize}

\noindent
For an extension $E/F$ of fields, the embedding $F^{\times} \inj E^{\times}$ 
induces a homomorphism 
\begin{equation}
	\label{eq:j}
j_{E/F}:K_2(F) \to K_2(E).
\end{equation}
If the extension $E/F$ is finite, then we also have the norm map
\begin{equation}
\label{eq:Norm}
N_{E/F}:K_2(E) \to K_2(F).
\end{equation}

\noindent 
For  a complete discrete valuation field $K$, 
we define 
\begin{itemize}
\item $\vK:K^{\times} \to \Z$ :\,the valuation of $K$, 
\item $\OK:=\set{f \in K | \vK(f)\ge 0}$ :\, the valuation ring of $K$,
\item $\mK:=\set{f \in K | \vK(f) >0}$ :\, the maximal ideal of $\OK$,  
\item $k_K:= \OK/\mK $ :\, the residue field of $K$,  
\item $U_K := \OKt$ :\, the group of units in $\OK$, 
\item $U_K^n := 1 + \mK^n$ :\, the higher unit groups, and
\item $\d_K : K_2(K) \to k_K^{\times}$ :\, the \textbf{tame symbol map} 
defined by 
\begin{equation}
\label{eq:tame_symbol}
 \d_K(\set{f,g}) := (-1)^{v_K(f)v_K(g)}f^{v_K(g)}g^{-v_K(f)} \bmod \m_K, 
\end{equation} 
for $\set{f,g}\in K_2(K)$, 
and 
\item $U^1K_2(K)$: the subgroup of $K_2(K)$ generated by 
the image of $U_K^1 \times K^{\times}$ by the 
symbol map 
$K^{\times} \times K^{\times} \to K_2(K); (x,y) \mapsto \set{x,y}$. 
\end{itemize}

%

\subsection*{Acknowledgments}
This work was supported by KAKENHI 	17K05174.

\section{Id\`ele class groups}
\label{sec:class groups}

In this section, we use the following notation: 

\begin{itemize}
	\item $X$\,: a variety over a field $F$, 
	\item $X_0$\,: the set of closed points in $X$,   
	\item $\CuX$\,: the set of the normalizations $\phi: C \to X$ 
	of a curve in $X$ (\Cf Notation). 
For simplicity, we often refer to the domain $C$ of the normalization 
$\phi:C\to X$ in $\CuX$ 
as an element of $\CuX$ and write $C\in \CuX$.  
	\item $\Fx$\,: the residue field at $x \in X_0$.
\end{itemize}
For each $\phi:C\to X \in \CuX$, 
\begin{itemize}
	\item $\Cbar$\,: the regular compactification of $C$ 
	which is the smooth and proper curve over $F$ containing $C$ as a dense open subscheme 
(by a resolution of singularities), 
	\item $\phibar:\Cbar \to \Xbar$\,: the canonical extension of $\phi$,  
	\item $\Cinf := \Cbar \ssm C$, and
	\item $\FCx= \operatorname{Frac}(\mathscr{O}_{C,x}^{\wedge})$\,: the completion of the function field $\FC$ of $C$ at $x \in \Cbar_0$. 
\end{itemize}

\subsection*{Id\`ele group}

\begin{definition}
\label{def:IX}
The \textbf{id\`ele group} of $X$ is defined by  
\[
\IX = \bigoplus_{x\in X_0} \Fx^{\times} \oplus 
  \bigoplus_{C \in \CuX}\bigoplus_{x\in \Cinf}K_2(\FCx). 
\]
\end{definition}


\begin{lem}
\label{lem:I}
	Let\/ $\vp: X'\to X$ be a morphism of varieties over $F$. 
%
Then, there is a canonical homomorphism 
		$\vp_{\ast}: I(X') \to I(X)$. 
\end{lem}
\begin{proof}
We define the homomorphism 
\[
\vp_{\ast}:\IXp\to \IX; (\xi_{x'})_{x'} \mapsto (\vp_{\ast}^{x'\to x}(\xi_{x'})_x)
\] 
by defining $\vp_{\ast}^{x'\to x}$ for each component as follows:

\begin{itemize}
\item For $x' \in X'_0$ and $x = \vp(x') \in X_0$, 
the norm homomorphism of the field extension 
$\Fxp/\Fx$ gives 
\[
\vp_{\ast}^{x'\to x}:= N_{\Fxp/\Fx}:\Fxpt\to \Fxt.
\]

\item For $C' \in\CuXp$, and $x' \in \Cpinf$, 
when $C'$ is the normalization of a curve $Z'$ in $X'$, 
we denote by 
$Z = \overline{\vp(Z')} \subset X$ the scheme theoretic closure of $\vp(Z')$, and 
$C \to Z$ is the normalization of $Z$. 
The morphism $\vp$ induces  
$\vpbar_C:\Cpbar \to \Cbar$.

\begin{itemize}
\item If $x := Z$ is a closed point of $X$, then we have 
\[
\xymatrix{
	\vp_{\ast}^{x'\to x}:K_2(\FCpxp) \ar[r]^-{\partial_{F(C')_{x'}}} & \Fxpt \ar[r]^{N_{\Fxp/\Fx}} &\Fxt,
	}
\]
where $\partial_{F(C')_{x'}}$ is the tame symbol map (defined in \eqref{eq:tame_symbol}). 

\item If $Z$ is  a curve in $X$ (and thus $C\in\CuX$), 
and $x := \vpbar_C(x')$ is in $C_0$, then  
we use 
\[
\xymatrix{
	\vp_{\ast}^{x'\to x}: K_2(\FCpxp) \ar[r]^-{\partial_{F(C')_{x'}}} & \Fxpt \ar[r]^{N_{\Fxp/\Fx}} & \Fxt.
	}
\]
\item 
If $Z$ is a curve in $X$ (and hence $C \in \CuX$) and 
$x := \vpbar_C(x')$ is in $\Cinf$, 
then $\vp$ induces the (finite) morphism $\vp_C:C' \to C$ of curves 
and this gives the norm map 
\[
\xymatrix{
	\vp_{\ast}^{x'\to x}:=N_{\FCpxp/\FCx}: K_2(\FCpxp) \ar[r] & K_2(\FCx)
	}
\]
of the Milnor $K$-groups (\Cf \eqref{eq:Norm}). 
\end{itemize}
\end{itemize}
These homomorphisms define	$\vp_{\ast} :\IXp \to \IX$.

%
%
%
%
\end{proof}

\subsection*{Id\`ele class group}
For each $\phi:C\to X \in \CuX$, 
we define a homomorphism 
\[\d_C:K_2(\FC) \to I(C) = \bigoplus_{x\in C_0} \Fxt \oplus \bigoplus_{x\in \Cinf}K_2(\FCx)
; \xi \mapsto (\d_{C,x}(\xi))_x
\]
 as follows:
%
%
\begin{itemize}
\item 
For $x \in C_0$, 
the inclusion $\FC \inj \FCx$ 
induces 
$j_{\FCx/\FC}:K_2(\FC) \to K_2(\FCx)$ (\Cf \eqref{eq:j}). 
We have the composite  
\[
\xymatrix{
\partial_{C,x}:K_2(\FC) \ar[r]^-{j_{\FCx/\FC}} & K_2(\FCx) \ar[r]^-{\partial_{\FCx}} & \Fxt,
} 
\]
where $\partial_{\FCx}$ is the tame symbol map \eqref{eq:tame_symbol}.

\item For $x \in C_{\infty}$, 
the inclusion $\FC \inj \FCx$ 
induces 
\[
\xymatrix{
\partial_{C,x}:= j_{\FCx/\FC}: K_2(\FC) \ar[r] & K_2(\FCx).
}
\] 
\end{itemize}

\sn
These homomorphisms give the required 
$\d_C:K_2(\FC)\to \IC$. 
Composing with the sum of 
$\phi_{\ast}:\IC \to \IX$ (Lem.~\ref{lem:I}, (i)), 
we obtain a homomorphism  
\[
\vcenter{
\xymatrix{
  \displaystyle \d:\bigoplus_{C\in\CuX} K_2(\FC) \ar[r]^-{\oplus \partial_C} &\displaystyle \bigoplus_{C\in \CuX} \IC \ar[r]^-{\sum_C \phi_{\ast}} &  \IX.
  }
}
\]

\begin{definition}
\label{def:CX}
The cokernel 
\[
  C(X/F) = \Coker\left(\d:\bigoplus_{C \in\CuX} K_2(\FC) \to  \IX\right)
\] 
of 
$\d$ defined above is called the \textbf{id\`ele class group} of $X$. 
When the base field $F$ is not particularly important, we also write as $\CX$ for simplicity. 
\end{definition}

When $X$ is a projective smooth variety over $F$, 
a normalization $\phi:C\to X$ in $\CuX$ 
extends to $\phibar:\Cbar \to X$ which is also in $\CuX$.  
Thus, $C=\Cbar$. 
In this case, we have 
$\CX = \SKX$ in terms of the algebraic $K$-theory (\Cf \cite{Bloch81}).


\begin{lem}
\label{lem:C}
	Let $\vp:X'\to X$ be a morphism of varieties over $F$. 
Then, there is a canonical homomorphism 
	$\vp_{\ast}: C(X') \to C(X)$. 
\end{lem}
\begin{proof}
We show that $\vp_{\ast}:\IXp\to \IX$ (Lem.~\ref{lem:I}) induces 
$\vp_{\ast}: \CXp \to \CX$. 
Take the normalization $\phi':C'\to X' \in\CuXp$ 
of a curve $Z'$ in  $X'$.   
\begin{itemize}
	\item If the scheme theoretic closure of the image $Z = \ol{\vp(Z')}$ is a curve in $X$, 
	then 
	we denote by $\phi:C\to Z \inj X$ in $\CuX$ the normalization of $Z$, 
	and $\vp_{C}:C'\to C$ the morphism induced from $\vp$. 
	Consider the following diagram:
	\begin{equation}
		\label{eq:func}
		\vcenter{
	\xymatrix{
	K_2(\FCp) \ar[d]_{N_{\FCp/\FC}}\ar[r]^-{\partial_{C'}} & \ICp \ar[r]^{\phi'_{\ast}}\ar[d]^{(\vp_{C})_{\ast}} & \IXp\ar[d]^{\vp_{\ast}} \\ 
	K_2(\FC) \ar[r]^-{\partial_{C}}  & \IC \ar[r]^{\phi_{\ast}} & \IX, 
	}}
	\end{equation}
	where the left vertical map is the norm map.
	Since the norm maps and the tame symbol maps are commutative, 
	the left square of the diagram \eqref{eq:func} is commutative. 
	The commutativity of the right square in \eqref{eq:func} follows from Lem.~\ref{lem:I}. 
	 
	\item If the image $x = \vp(Z)$ is in $X_0$, 
	then the curve $C'$ is defined over $F':= F(x)$. 
	By the Weil reciprocity law on the Milnor $K$-groups (\Cf \cite{BT73}, Sect.~5),  
	the sequence 
\[
	\vcenter{
	\xymatrix{
	K_2(\FCp) \ar[r]^-{\d_{\Cpbar}} & \displaystyle \ICpbar = \bigoplus_{x' \in \Cpbar_0}\Fxpt \ar[r]^-{N_{\Cpbar}} & (F')^{\times}
	}}
\]
	is a complex, 
	where 
	$N_{\Cpbar} = \sum_{x'}N_{F(x')/F'}$  
	is the sum of the norm maps.
	We have the following diagram
	\begin{equation}
		\label{eq:func2}
	\vcenter{
	\xymatrix{
	K_2(\FCp) \ar[r]^-{\d_{C'}}\ar[d] & \ICp \ar[d]^-{ ({\vp_{x}})_{\ast}} \ar[r]^-{\phi'_{\ast}} & \IXp \ar[d]^{\vp_{\ast}} \\
	0 \ar[r] & I(x) = \Fxt \ar[r]^-{(\iota_x)_{\ast}} & \IX,  
	}}
	\end{equation}
	where $\vp_{x}:C' \to x$ is induced from $\vp$ 
	and the right horizontal map is given by the closed immersion 
	$\iota_x: x \inj X$. 
	The composite $(\vp_x)_{\ast} \circ \d_{C'} = N_{\Cpbar}\circ \d_{\Cpbar} = 0$ makes the above diagram commutative.
\end{itemize}
From the commutative diagrams \eqref{eq:func} and \eqref{eq:func2}, 
the map $\vp_{\ast}:\IXp \to \IX$ defines 
$\vp_{\ast}:\CXp \to \CX$. 
%
\end{proof}

%

By Lem.~\ref{lem:C}, the structure map $\gamma: X \to \Spec(F)$ induces 
\begin{equation}
\label{eq:NX}
	N_X:= \gamma_{\ast}: \CX \to C(\Spec(F)) = F^{\times}.
\end{equation}

\subsection*{Tame id\`ele class group}
We introduce a quotient $\CtX$ of $\CX$ which comes to classify \textit{tame coverings}   through the reciprocity map \eqref{eq:rhot}.  

\begin{definition}
\label{def:CtX}
The cokernel 
\[
C^t(X) = C^t(X/F) =  \Coker\left( \bigoplus_{C\in \CuX}\bigoplus_{x\in \Cinf} U^1K_2(F(C)_x) \to \CX \right)
\]
is called the \textbf{tame id\`ele class group} of $X$, 
where 
the map above 
is induced from the inclusion $U^1K_2(F(C)_x) \inj K_2(F(C)_x)$ 
for $C \in \CuX$ and $x \in C_{\infty}$. 
\end{definition}

\begin{lem}
\label{lem:perf}
Let $\wt{F} = F^{-p^{\infty}}$ be the perfect closure of $F$ 
with $\Char(F) = p$, 
that is, 
the field adjoined with all $p^r$-th roots to $F$ for all $r\ge 1$. 
Then, we have an isomorphism  
\[
\Ct(X/F)/l^m \isomto \Ct(X\otimes_F \wt{F}/\wt{F})/l^m,
\] 
for any prime number $l\neq p$ and $m\in \Z_{>0}$.
\end{lem}
\begin{proof}
We denote by $\wt{X} = X\otimes_F \wt{F}$ 
the base change to $\wt{F}$ of $X$. 
Since the extension $\wt{F}/F$ is purely inseparable, 
the projection $\wt{X} \to X$ is a homeomorphism (\cite{LiuAG}, Prop.~3.2.7 (c)). 
In particular, there are bijections $X_0 \simeq \wt{X}_0$ and $\Cu(X) \simeq \Cu(\wt{X})$. 
Consider the following commutative diagram:  
\begin{equation}
	\label{eq:kF}
\vcenter{
\entrymodifiers={!! <0pt, .8ex>+}
\xymatrix@R=5mm@C=3mm{
\displaystyle \bigoplus_{C\in \CuX} K_2(\FC)\ar[d] \ar[r] &\displaystyle \bigoplus_{x\in X_0}F(x)^{\times} \oplus \bigoplus_{C\in \CuX}\bigoplus_{x\in \Cinf} K_2(\FCx)/U^1K_2(\FCx)\ar[d] \\
\displaystyle \bigoplus_{C\in \Cu(\wt{X})} K_2(\wt{F}(C)) \ar[r] &\displaystyle \bigoplus_{x\in \wt{X}_0}\wt{F}(x)^{\times} \oplus \bigoplus_{C\in \Cu(\wt{X})}\bigoplus_{x\in \Cinf} K_2(\wt{F}(C)_x)/U^1K_2(\wt{F}(C)_x),
}
}
\end{equation}
where the vertical maps are given by the inclusion map $F\inj \wt{F}$. 
Note that the cokernel of the horizontal maps in the diagram \eqref{eq:kF} are 
$\Ct(X/F)$ and $\Ct(\wt{X}/\wt{F})$.

For each $C\in \CuX$, the corresponding 
curve in $\wt{X}$ is $\wt{C} := C\otimes_F \wt{F} \in \Cu(\wt{X})$.  
The extension $\wt{F}(\wt{C}) = F(C)\wt{F}$ of $F(C)$ 
is a constant field extension. 
For each $r>0$, 
we have 
\[
N_{\FC F^{-p^r}/\FC} \circ j_{\FC F^{-p^r}/\FC} = [\FC F^{-p^r}:\FC]: 
K_2(\FC) \to K_2(\FC),
\] 
where $j_{\FC F^{-p^r}/\FC}$ is the map 
induced from the inclusion $\FC \inj \FC F^{-p^r}$ (\Cf \eqref{eq:j}). 
Applying $-\otimesZ \Z/l^m$, 
we have $K_2(\FC)/l^m \simeq K_2(\FC F^{-p^r})/l^m$.   
The left vertical map 
in the above diagram \eqref{eq:kF} becomes bijective 
after applying $-\otimes_{\Z} \Z/l^m$. 
Moreover, 
we have 
\begin{align*}	
K_2(\FCx)/U^1K_2(\FCx) &\simeq K_2(\Fx)\oplus \Fxt \quad (C\in \CuX, x\in\Cinf),\ \mbox{and}\\
K_2(\wt{F}(C)_x)/U^1K_2(\wt{F}(C)_x) &\simeq K_2(\wt{F}(x)) \oplus \wt{F}(x)^{\times} \quad(C\in \Cu(\wt{X}), x\in\Cinf)\ 
\end{align*}
(\Cf \cite{FV}, Chap.~IX, Prop.~2.2). 
In the same way as above, 
$\wt{F}(x\otimes_F \wt{F}) \simeq F(x)\wt{F}$ for each $x \in X_0$ or $x\in \Cinf$ for some $C\in \CuX$. 
The right vertical map in \eqref{eq:kF} is an isomorphism 
after applying $- \otimes_{\Z} \Z/l^m$. 
Hence, there is an isomorphism 
\[
\Ct(\wt{X}/\wt{F})/l^m \simeq \Ct(X/F)/l^m.
\]
\end{proof}

Corresponding to the tame id\`ele class group 
$\CtX$, 
we define 
the \textit{abelian tame fundamental group} of $X$ 
as a quotient of $\piabX$.

\begin{definition}
	\label{def:pit}
	The \textbf{abelian tame fundamental group} is 
	defined by  
	\[
	\pitabX = \pitab(X/F) =  
	\Coker\left(\bigoplus_{C \in \CuX}\bigoplus_{x \in \Cinf} I_{\FCx}^1 
	\to \piabX\right)
	\]
	where $I_{\FCx}^1$ is the $p$-Sylow subgroup of 
	the abelian Galois group $G_{\FCx}^{\ab}$ 
	for $\phi:C\to X\in \CuX, x\in \Cinf$, 
	and the map above is induced from 
	the composition $I_{\FCx}^1 \inj G_{\FCx}^{\ab} \inj \piabC \onto{\phi_{\ast}} \piabX$.
\end{definition}

\subsection*{Reciprocity map}
In the rest of this section, we consider 
\begin{itemize}
	\item $k$: a local field, 
	\item $X$: a \textit{smooth} variety over $k$, and
	\item $\gamma:X\to \Spec(k)$: the structure map.   
\end{itemize}
For the variety $X$, 
we introduce the reciprocity map 
\begin{equation}
\label{eq:rhoX}
  \rho_X: C(X) := C(X/k) \to \piabX.
\end{equation}
First, we define a group homomorphism 
\[
\wt{\rho}_X:\IX\to \piabX; (\xi_x)_x \mapsto \sum_x \wt{\rho}_{X,x}(\xi_x)
\] 
by introducing  
the following homomorphisms $\wt{\rho}_{X,x}$: 
\begin{itemize}
\item For $x \in X_0$, 
we denote by $\iota_x:x \inj X$ the closed immersion. 
The reciprocity map 
$\rho_{\kx}: \kxt \to G_{\kx}^{\ab}$ of 
the local class field theory (for the residue field $\kx$) gives 
\[
\xymatrix{
\wt{\rho}_{X,x}:\kx^{\times} \ar[r]^-{\rho_{\kx}}& G_{\kx}^{\ab} = \piab(x) \ar[r]^-{(\iota_x)_{\ast}} &\piabX.
}
\] 

\item For $\phi:C\to X \in\CuX$ and $x\in\Cinf$, 
the completion $\kCx$ is a $2$-dimensional local field (\Cf \cite{Hir17}, Def.~2.1) 
with residue field $\kx$. 
The reciprocity map $\rho_{\kCx}:K_2(\kCx) \to G_{\kCx}^{\ab}$ of the 2-dimensional local class field theory (for $\kCx$) induces 
\[
\xymatrix{
\wt{\rho}_{X,x}:K_2(\kCx) \ar[r]^-{\rho_{\kCx}} & G_{\kCx}^{\ab} \subset G_{\kC}^{\ab} = \piab(\Spec(\kC) \ar[r]^-{(j_C)_{\ast}} & \piabX,
}
\]
where the last homomorphism 
is induced from 
$j_C: \Spec(\kC) \inj C \onto{\phi}X$.
\end{itemize}

For a morphism $\vp:X'\to X$ of varieties over $k$, 
from the construction of $\wt{\rho}_X$ above 
and some properties of 
the reciprocity map of $2$-dimensional local class field theory 
(\Cf \cite{Hir17}, Prop.~2.2), 
we have the following commutative diagram: 
\begin{equation}
\label{eq:rho-f} 
\vcenter{
\xymatrix{
\IXp \ar[d]^{\vp_{\ast}} \ar[r]^{\wt{\rho}_{X'}} & \piab(X') \ar[d]^{\vp_{\ast}} \\ 
\IX \ar[r]^-{\wt{\rho}_X} & \piabX.
}}
\end{equation}

\begin{lem}
\label{lem:rho}
$\wt{\rho}_X:\IX \to \piabX$ factors through $\CX$. 
\end{lem}
 \begin{proof}
 \textbf{(The case \mbox{{\boldmath$\dim(X) =0, 1$}})}\ 
If $\dim X =0$, then $I(X) = k(X)^{\times} = \CX$ and there is nothing to show.  
In the case where $X$ is a (smooth) curve,   
the assertion follows from Sect.\ 2 of \cite{Hir10}.  
The main ingredient of the discussion in \opcit\ is   
the reciprocity law of $\kX$ (\cite{Sai85}, Chap.~II, Prop.~1.2).

\smallskip
\noindent
\textbf{(The case \mbox{{\boldmath$\dim(X) > 1$}})}\  
For a general variety $X$, 
for each $\phi:C\to X\in\CuX$, 
the maps $\wt{\rho}_C$ and $\wt{\rho}_X$ defined above  give the diagram 
\[
\xymatrix@R=5mm{
	K_2(\kC)\ar[r]^-{\d_C} & I(C) \ar[d]^{\phi_{\ast}} \ar[r]^{\wt{\rho}_C} & \piabC\ar[d]^{\phi_{\ast}} \\
	 & \IX \ar[r]^{\wt{\rho}_X} & \piabX
}
\]
with $\wt{\rho}_C\circ \d_C =0$  
from the case of $\dim X =1$ discussed above. 
Since the above diagram \eqref{eq:rho-f} is commutative, 
the assertion $\wt{\rho}_X \circ \d = 0$ follows. 
\end{proof}

\begin{definition}
\label{def:rho}
The induced map $\rho_X:\CX \to \piabX$ from 
$\wt{\rho}_X$ by Lem.~\ref{lem:rho} is 
called the \textbf{reciprocity map} for $X$. 	
\end{definition}

The commutative diagram \eqref{eq:rho-f} and 
Lem.~\ref{lem:rho} say that 
for each $\vp:X'\to X$, the reciprocity maps make the 
following diagram commutative:
\[
\vcenter{
\xymatrix{
\CXp \ar[d]^{\vp_{\ast}} \ar[r]^{\rho_{X'}} & \piab(X') \ar[d]^{\vp_{\ast}} \\ 
\CX \ar[r]^-{\rhoX} & \piabX 
}}
\]

\noindent
In particular, the structure map $\gamma:X \to \Spec(k)$ gives 
the following commutative diagram: 
\begin{equation}
\label{eq:VX}
\vcenter{
\xymatrix{
  0 \ar[r] & \CXg \ar[r]\ar[d]^{\rhoX}& \CX \ar[r]^-{N_X}\ar[d]^{\rhoX} &  k^{\times} \ar[d]^{\rho_k}\\
  0\ar[r]  & \piabXg \ar[r] & \piabX \ar[r]^-{\gamma_{\ast}} & G_k^{\ab},
}}
\end{equation}
where 
$N_X$ is defined in \eqref{eq:NX} and 
$\rho_k = \rho_{\Spec(k)}$ is the reciprocity map of $k$ and 
the groups $\CXg$ and $\piabXg$ are defined by the exactness of the horizontal rows. 
From the $2$-dimensional local class field theory (\cite{Kat87}, see also \cite{Hir17}, Prop.~2.5), 
the reciprocity map $\rhoX$ induces a map 
from 
the tame class group (Def.~\ref{def:CtX}) 
to the tame fundamental group (Def.~\ref{def:pit}) 
as the following commutative diagram indicates: 
\begin{equation}
	\label{eq:rhot}
	\vcenter{
	 \xymatrix{
	 \CX  \ar@{->>}[d] \ar[r]^{\rhoX} & \piabX \ar@{->>}[d] \\
	 \CtX \ar@{-->}[r] & \pitabX,
	 } }
\end{equation}
where the vertical maps are the quotient maps. 
The induced map $\CtX \to \pitabX$ is also denoted by $\rho_X$.
As in \eqref{eq:VX}, we have the following commutative diagram with exact rows: 
\begin{equation}
\label{eq:VtX}
\vcenter{
\xymatrix{
  0 \ar[r] & \CtXg \ar[r]\ar[d]^{\rhoX}& \CtX \ar[r]^-{N_X}\ar[d]^{\rhoX} &  k^{\times} \ar[d]^{\rho_k}\\
  0\ar[r]  & \pitabXg \ar[r] & \pitabX \ar[r]^-{\gamma_{\ast}} & G_k^{\ab},
}}
\end{equation}

\section{Products of curves}
%

\subsection*{Somekawa $K$-groups}
First, we recall the definition of the Mackey products and that of the Somekawa $K$-groups 
following \cite{RS00}, \cite{Yam09} and \cite{KY13}:  
Recall that 
a \textbf{Mackey functor} $A$ over a field $F$ 
is a contravariant 
functor from the category of \'etale schemes over $F$ 
to that of abelian groups 
equipped with a covariant structure 
for finite morphisms satisfying some conditions 
(for the precise definition, see \cite{RS00} Section 3, or \cite{KY13}, Sect.~2).
For a Mackey functor $A$ over $F$, 
we denote by $A(E)$ 
its value $A(\Spec(E))$  
for a field extension $E$ over $F$.

\begin{definition} 
	\label{def:otimesM}
For Mackey functors $A_1,\ldots , A_n$ over $F$, 
their \textbf{Mackey product} 
$A_1\otimesM \cdots \otimesM A_n$ 
is defined as follows: 
For any finite field extension $E/F$, 
\begin{equation}
	\label{eq:otimesM}
\left(A_1\otimesM \cdots \otimesM A_n \right) (E) 
 := \left(\bigoplus_{E'/E:\, \mathrm{finite}} A_1(E') \otimesZ \cdots \otimesZ A_n(E')\right) /R, 
\end{equation}
where $R$ is the subgroup generated 
by elements of the following form: 

\sn
\PF 
For any finite field extensions 
$E \subset E_1 \subset E_2$, and 
if $x_{i_0} \in A_{i_0}(E_2)$ and $x_i \in A_i(E_1)$ 
for all $i\neq i_0$, then 
\[
  j^{\ast}(x_1) \otimes \cdots \otimes x_{i_0} \otimes \cdots \otimes j^{\ast}(x_n) - x_1\otimes \cdots \otimes j_{\ast}(x_{i_0})\otimes \cdots \otimes x_n,
\]
where $j = j_{E_2/E_1}:\Spec(E_2) \to \Spec(E_1)$ is the canonical map. 
\end{definition}

For the Mackey product  
$A_1\otimesM \cdots \otimesM A_n$,  
we write $\set{x_1,\ldots,x_n}_{E/F}$ 
for the image of 
$x_1 \otimes \cdots \otimes x_n \in 
A_1(E) \otimes \cdots \otimes A_n(E)$ in the product 
$\left(A_1\otimesM \cdots \otimesM A_n\right)(F)$. 
For any field extension $E/F$, 
the canonical map $j=j_{E/F}:F\inj E$ induces  
the pull-back 
\[
  \Res_{E/F} := j^{\ast}: \left(A_1\otimesM \cdots \otimesM A_n\right)(F) \longrightarrow 
\left(A_1\otimesM \cdots \otimesM A_n\right)(E).
\] 
If the extension $E/F$ is finite, 
then the push-forward 
\[
  N_{E/F} := j_{\ast}: \left(A_1\otimesM \cdots \otimesM A_n\right)(E) \longrightarrow 
\left(A_1\otimesM \cdots \otimesM A_n\right) (F)
\] 
is given by 
$N_{E/F}(\set{x_1,\ldots ,x_n}_{E'/E}) = \set{x_1,\ldots, x_n}_{E'/F}$ 
on symbols. 
%
%

\begin{definition}
\label{def:Som}
Let $\F_1,\ldots \F_n$ be homotopy invariant Nisnevich sheaves with transfers over $F$. 
By considering these sheaves as Mackey functors (\cite{KY13}, Sect.~2), 
the \textbf{Somekawa $K$-group} $K(F;\F_1,\ldots ,\F_n)$ attached to $\F_1,\ldots ,\F_n$ 
is defined by a quotient 
\begin{equation}
	\label{eq:Som}
  K(F;\F_1,\ldots ,\F_n) := \left\{\left(\F_1 \otimesM \cdots \otimesM \F_n \right) (F)\right\}/R, 
\end{equation}
where $R$ is a subgroup 
which produces 
``the Weil reciprocity law'' (for the precise definition, see \cite{KY13} Def.~5.1). 
\end{definition}

For semi-abelian varieties $A_1,\ldots , A_n$ over $F$ 
by considering them as Nisnevich sheaves (\Cf \cite{KY13}, Sect.~2.15), 
this $K$-group coincides with that of defined in \cite{Som90} 
(\cite{KY13}, Rem.~5.2).

In the rest of this section, we use 
\begin{itemize}
	\item $F$: a field with $\Char(F) = p\ge 0$, 
	\item $X_i \subset \Xbar_i\,:$ open curves over a field $F$ (\Cf Notation)  
	with $X_i(F) \neq \emptyset$ for $i = 1,\ldots , n$,  
	\item $X = X_1\times \cdots \times X_n \subset \Xbar = \Xbar_1 \times \cdots \times \Xbar_n$, and 
	\item $\Jac(X_i):$ the generalized Jacobian variety 
	of $X_i$. 
\end{itemize}

\begin{lem}[\Cf \cite{Yam09}, Thm.\ 2.2]
\label{lem:dec}
	Assume that $F$ is a \textbf{perfect} field.  
%
	Then, there is an isomorphism  
	\[
	   \CtX \simeq \bigoplus_{r = 0}^n\, \bigoplus_{1 \le i_1 < \cdots < i_r \le n} K(F;\Jac(X_{i_1}),\ldots , \Jac(X_{i_r}), \Gm),
	\]
	where $\CtX$ is the tame id\`ele class group of $X$ (Def.~\ref{def:CtX}).
\end{lem}
\begin{proof}
The  tame id\`ele class group $C^t(X)$ 
is isomorphic to \emph{Wiesend's tame ideal class group} $C_1(X)$ 
in the sense of \cite{Yam13} (see also \cite{Yam13}, Rem.~5.5). 
From \cite{Yam13}, Thm.~1.3, 
we obtain
\begin{equation}
\label{eq:DM}
	C_1(X) \simeq \Hom_{\DM}(\Z(-1)[-1],M(X)) \simeq  
\Hom_{\DM}(\Z,M(X)(1)[1]),
\end{equation}
where $\DM = \DM_{\mathrm{Nis}}^{\mathrm{eff},-}(F)$ 
is the triangulated tensor category of Voevodsky's motivic complexes (\cite{V}).  
Following \cite{V} and \cite{KY13}, Sect.~9, 
we define $\hNis(X_i) := \hNis(L(X_i))$. 
By \cite{KY13}, Thm.~12.3, the far right 
in \eqref{eq:DM} has a description
\begin{equation}
	\label{eq:KSconj}
\Hom_{\DM}(\Z,M(X)(1)[1]) \simeq 
K(F; \hNis(X_1),\ldots,\hNis(X_n),\Gm),
\end{equation}
where the right is the Somekawa $K$-group
for homotopy invariant Nisnevich sheaves with transfers (\cite{KY13}, Def.~5.1). 
From \cite{KY13}, Lem.\ 11.2, 
$\hNis(X_i)$ coincides with the presheaf of relative Picard groups with respect to $(\Xbar_i,X_i)$. From the assumption $X_i(F)\neq \emptyset$, we have 
\[
\hNis(X_i) \simeq \Z \oplus \Jac(X_i).
\]
Here, $\Jac(X_i)$ is the generalized Jacobian variety 
regarded as a Nisnevich sheaf with transfers (\Cf \cite{KY13}, 2.15). 
Since $K(F; \hNis(X_1),\ldots,\hNis(X_n),\Gm)$ is a quotient 
of the Mackey product (as noted in \eqref{eq:Som}), 
we have 
\begin{align*}
&K(F; \hNis(X_1),\ldots,\hNis(X_n),\Gm) \\
&\simeq K(F;\Z , \hNis(X_2), \ldots ,\hNis(X_n),\Gm) \oplus 
K(F;\Jac(X_1) , \hNis(X_2), \ldots ,\hNis(X_n),\Gm) \\
&\simeq 
K(F;\hNis(X_2), \ldots ,\hNis(X_n),\Gm) \oplus 
K(F;\Jac(X_1) , \hNis(X_2), \ldots ,\hNis(X_n),\Gm). 
\end{align*}
Inductively, we obtain the decomposition of the $K$-group in \eqref{eq:KSconj}:  
\[
K(F; \hNis(X_1),\ldots,\hNis(X_n),\Gm) \simeq 
\bigoplus_{r=0}^n \bigoplus_{1 \le i_1 < \cdots < i_r \le d} K(F; \Jac(X_{i_1}),\ldots, \Jac(X_{i_r}), \Gm).
\]
The assertion follows from these isomorphisms. 
\end{proof}

\begin{lem}
\label{lem:dec2}
	Let $l$ be a prime number $l\neq p$ and $m\in \Z_{>0}$. 
	We assume that $F$ is perfect, and the following condition:
	\begin{description}
		\item [{$(\mathbf{Surj})\,$}] 
	 For any finite extensions $F''\supset F'\supset F$, 
the norm map 
\[
\Ct(X_i\otimes_F F'')^{\geo}/l^m \to \Ct(X_i \otimes_F F')^{\geo}/l^m
\] 
is surjective for each $1\le i \le n$.
	\end{description} 
	Then, we have 
	\[
	\CtXg/l^m \simeq \bigoplus_{i=1}^n \Ct(X_i)^{\geo}/l^m.
	\]
\end{lem}
\begin{proof}
From Lem.~\ref{lem:dec} and $K(F;\Gm) \simeq F^{\times}$, we have 
\begin{align*}
   \CtXg &\simeq \bigoplus_{r = 1}^n\, \bigoplus_{1 \le i_1 < \cdots < i_r \le n} K(F;\Jac(X_{i_1}),\ldots , \Jac(X_{i_r}), \Gm)\quad \mbox{and} \\
   \Ct(X_i)^{\geo} &\simeq K(F;\Jac(X_{i}), \Gm)\quad \mbox{for $i=1,\ldots,n$.}		   
\end{align*}
It is enough to show 
\[
K(F;\Jac(X_1), \Jac(X_2), \ldots ,\Jac(X_t),\Gm)/l^m = 0
\] 
for $t\ge 2$. 
Without loss of generality, we may assume $t=2$.
We consider $M := K(-;\Jac(X_2),\Gm)$ as a Mackey functor over $F$ defined by 
\[
E/F \mapsto M(E) = K(E;\Jac(X_2),\Gm).
\] 
From the very definition of the $K$-group (Def.~\ref{def:Som}), 
we have the surjective map
\[
\left(\Jac(X_1)\otimesM M\right)(F) \surj K(F; \Jac(X_1),\Jac(X_2),\Gm).
\]  
For any element of the form $\set{x', y'}_{F'/F} \in \left(\Jac(X_1)\otimes M\right)(F)$, 
there exists a finite (separable) extension $F''/F'$ such that 
$\Res_{F''/F'}(x') = l^m x''$ for some $x'' \in \Jac(X_1)(F'')$. 
From the assumption (\textbf{Surj}) and $M(F') \simeq \Ct(X_2\otimes_F F')^{\geo}$, 
the norm map 
$N_{F''/F'}:M(F'')/l^m \surj M(F')/l^m$ 
is surjective. 
Hence, there exists 
$y'' \in M(F'')$ such that $N_{F''/F'}(y'') = y' \bmod l^m$. 
Namely, there exists $z' \in M(F')$ such that 
$N_{F''/F'}(y'') = y' + l^m z'$. 
From the ``projection formula'' \PF\ 
in Def.~\ref{def:otimesM} (the definition of the Mackey product \eqref{eq:otimesM}), we have 
\begin{align*}
	\set{x', y'}_{F'/F}
	& = \set{x', N_{F''/F'}(y'') - l^m z'}_{F'/F}\\
	& = \set{x', N_{F''/F'}(y'')}_{F'/F} - \set{x',l^mz'}_{F'/F}\\
	&\stackrel{(\textbf{PF})}{=} \set{\Res_{F''/F'}(x'), y''}_{F''/F} - l^m \set{x', z'}_{F'/F}\\
 	&= \set{l^m x'',y''}_{F''/F}  - l^m \set{x', z'}_{F'/F}\\
	&= l^m \left(\set{x'',y''}_{F''/F} -  \set{x', z'}_{F'/F}\right).
\end{align*}
This implies  
$\left(\Jac(X_1)\otimesM M\right)(F)/l^m = K(F; \Jac(X_1),\Jac(X_2),\Gm)(F)/l^m = 0$.
\end{proof}

\subsection*{Proof of the main theorem}

Now, we suppose 
the base field $F=k$ is a local field (with $\Char(k) = p\ge 0$),  
and devote to prove the following theorem. 
%

\begin{thm}
\label{thm:main}
Assume the following conditions 
\Red or \Split for each $\Xbar_i$: 
\begin{description}
	\item[\Red] the Jacobian variety $\Jac(\Xbar_i)$ of $\Xbar_i$ has potentially good 
reduction,\  
	\item[\Split] the special fiber of the connected component of the N\'eron model 
of $\Jac(\Xbar_i)$ is an extension of an abelian variety 
by a split torus. 
\end{description}

\noindent
Then, the kernel $\Ker(\rhoX)$ is the maximal $l$-divisible subgroup of $\CX$ 
for all prime number $l\neq p$.  
\end{thm}

The theorem above is proved by the following results on the reciprocity map $\rho_{X_i}$ 
for the curve $X_i$. 

\begin{thm}[\cite{Hir10} and \cite{Hir17}]
\label{thm:curve}
For each open curve $X_i\subset \Xbar_i$ over $k$, we have 
\begin{enumerate}
	\item $\piab(X_i)/\ol{\Im(\rho_{X_i})} \simeq \Zhat^{r_i}$, 
	where $\ol{\Im(\rho_{X_i})}$ is the topological closure of the image $\Im(\rho_{X_i})$ 
	and $r_i = r(\Xbar_i)$ is the rank of $\Xbar_i$ (\Cf \cite{Sai85}, Chap.~II, Def.~2.5). 
	\item $\Ker(\rho_{X_i})$ is $l$-divisible for any prime $l\neq p$. 
\end{enumerate}	
\end{thm}

To prove Thm.~\ref{thm:main}, we prepare some lemmas. 
\begin{lem}
\label{lem:div}
	Let $K$ be a complete discrete valuation field of characteristic $p\ge 0$. 
	Then, $U^1K_2(K)$ is $l$-divisible for all prime $l\neq p$. 	
\end{lem}
\begin{proof}
	Recall that $U^1K_2(K)$ is generated by $\set{U^1_K,K^{\times}} \subset  K_2(K)$. 
	As $K$ is complete, the unit group $U_K^1$ is $l$-divisible (\cite{FV}, Chap.\ I, 
	Cor.\ 5.5). Therefore, 
	$U^1K_2(K)$ is also $l$-divisible.
\end{proof}

\begin{lem}
\label{lem:key}
	For any $m\in \Z_{>0}$ and a prime $l\neq p$, 
	$\CX/l^m \to \piabX/l^m$ induced from $\rho_X$ is injective. 
\end{lem}
\begin{proof}
Consider the following commutative diagram with exact rows (\Cf \eqref{eq:VX}): 
\[
\xymatrix{
0\ar[r] & \CXg\ar[d]^{\rho_{X}} \ar[r]& \CX\ar[d]^{\rho_X} \ar[r]^{N_X} & k^{\times} \ar[d]^{\rho_k} \ar[r] & 0\\ 
0 \ar[r] & \piabXg \ar[r] & \piabX \ar[r] & G_k^{\ab} \ar[r] & 0.
}
\] 
The existence of a $k$-rational point of $X$ implies that  the horizontal sequences 
in the above diagram 
are split. 
By local class field theory, the induced homomorphism 
$k^{\times}/l^m \to G_k^{\ab}/l^m$ from $\rho_k$ 
is injective (in fact, bijective). 
Thus, to show the assertion  
it is enough to prove that 
$\rho_{X}:\CXg/l^m \to \piabXg/l^m$ 
(we use the same notation as $\rho_X:\CX \to \piabX$ for simplicity) 
is injective.

From the $2$-dimensional local class field theory and 
the commutative diagrams \eqref{eq:VX} and \eqref{eq:VtX}, 
we have 
\[
\vcenter{
\entrymodifiers={!! <0pt, .8ex>+}
\xymatrix@R=5mm{
\displaystyle{\bigoplus_{C \in \CuX}\bigoplus_{x\in C_{\infty}} U^1K_2(k(C)_x)} 
\ar[d]^-{\rho_{\kCx}} \ar[r] & \CXg\ar[d]^{\rho_X} \ar[r] & \CtXg \ar[d]^{\rho_X}\ar[r] & 0, \\
\displaystyle{\bigoplus_{C \in \CuX}\bigoplus_{x\in C_{\infty}} I_{k(C)_x}^1} \ar[r] & \piabXg \ar[r]&  \pitabXg \ar[r] & 0. 
}
}
\]
From Lem.~\ref{lem:div}, 
we have $U^1K_2(\kCx)/l^m = 0$ for each $C \in \CX$ and $x\in \Cinf$ 
so that $\CXg/l^m \isomto \CtXg/l^m$. 
The wild inertia subgroup
$I_{\kCx}^1$ is pro-$p$ so that 
we have 
$\piabX/l^m\isomto \pitabX/l^m$.
Hence, the assertion is reduced to showing that 
$\rho_{X}:\CtXg/l^m \to \pitabXg/l^m$ is injective.

Let $\wt{k} = k^{-p^{\infty}}$ be the perfect closure of $k$. 
We denote by $\wt{X} = X\otimes_k \wt{k}$ 
the base change to $\wt{k}$ of $X$. 
We also denote by 
\[
\xymatrix{
  \Ct(\wt{X}/\wt{k})^{\geo} = \Ker\Big( \Ct(\wt{X}/\wt{k}) \ar[r]^-{N_{\wt{X}}} & \wt{k}^{\times}\Big) 
 }
\]
as in \eqref{eq:VX}. 
From Lem.~\ref{lem:perf} and 
$k^{\times}/l^m \simeq \wt{k}^{\times}/l^m$ 
we have an isomorphism  
$\Ct(X/k)^{\geo}/l^m \isomto \Ct(\wt{X}/\wt{k})^{\geo}/l^m$. 
It is well-known that we have $\pi_1(\X) \isomto \pi_1(X)$ (\Cf \cite{SGA1}, Exp.~IV, 
Proof of Thm.\ 6.1).  
we obtain $\rho_{\X}$ below: 
\[
\xymatrix{
\Ct(X/k)^{\geo}/l^m \ar[r]^{\rho_X} & \pitabXg/l^m \\
\Ct(\X/\wt{k})^{\geo}/l^m\ar[u]^{\simeq} \ar@{-->}[r]^{\rho_{\X}} & \pitab(\X/\wt{k})^{\geo}/l^m\ar[u]^{\simeq},
}
\]
where $\pitab(\X/\wt{k})^{\geo} = \Ker(\pitab(\X) \to G_{\wt{k}}^{\ab})$. 
Thus the assertion is reduced to showing that $\rho_{\sX}$ is injective. 
In the following, we write $\Ct(\X)^{\geo} = \Ct(\X/\wt{k})^{\geo}$ 
and $\pitab(\X)^{\geo } := \pitab(\X/\wt{k})^{\geo}$ for simplicity. 
We prove that 
\[
\rho_{\X}:\Ct(\X)^{\geo}/l^m \to \pitab(\X)^{\geo }/l^m 
\] 
is injective. 

Recall that $r_i = r(\Xbar_i)$ is the rank of $X_i$ (\Cf Thm.~\ref{thm:curve}). 
From the assumption \Red\, or \Split, we have 
$r(\Xbar_i\otimes_k k') = r(\Xbar_i)$ for any finite extension $k'/k$ (\cite{Sai85}, Chap.~II, Thm.~6.2). 
Since $\rho_{X_i}$ is injective and $\Coker(\rho_{X_i}) \simeq \Z^{r_i}/l^m$ (Thm.~\ref{thm:curve}), 
the same results hold for 
\[
\rho_{\X_i}: \Ct(\X_i)^{\geo}/l^m \to \pitab(\X_i)^{\geo}/l^m, 
\]
where $\X_i = X_i \otimes_k \wt{k}$. 
%

\begin{claim}
\label{claim:Ct}
The condition {$(\mathbf{Surj})\,$} in Lem.~\ref{lem:dec2} holds.\ 
Namely, for any finite extensions $\wt{k}''\supset \wt{k}'\supset \wt{k}$, 
put $\X_i'  = \X_i \otimes_{\wt{k}} \wt{k}'$ and $\X_i'' = \X_i \otimes_{\wt{k}} \wt{k}''$. 
Then,  
the norm map 
\[
\Ct(\X_i'')^{\geo}/l^m \to \Ct(\X_i')^{\geo}/l^m
\] 
is surjective. 
\end{claim}
\begin{proof}
(We follow the argument of \cite{Yam09}, Prop.~1.7.)
Consider the following commutative diagram with exact rows:
\[
\xymatrix{
0\ar[r] & \Ct(\X_i'')^{\geo}/l^m \ar[d]\ar[r]^{\rho_{\X_i''}} &  \pitab(\X_i'')^{\geo}/l^m \ar@{->>}[d]\ar[r] & (\Z/l^m)^{r_i}\ar@{=}[d] \ar[r]& 0\\
0 \ar[r] & \Ct(\X_i')^{\geo}/l^m \ar[r]^{\rho_{\X_i'}} &\pitab(\X_i')^{\geo}/l^m \ar[r] & (\Z/l^m)^{r_i} \ar[r] & 0.
}
\]
Since the middle vertical map is surjective, 
the assertion follows from the above diagram. 
\end{proof}

From Lem.~\ref{lem:dec2}, 
we have the following commutative diagram: 
\[
\vcenter{
\entrymodifiers={!! <0pt, .8ex>+}
\xymatrix@C=10mm{
\Ct(\X)^{\geo}/l^m \ar[d]^{\simeq} \ar[r]^{\rho_{\X}} & \pitab(\X)^{\geo}/l^m\ar[d] \\
\displaystyle{\bigoplus_{i=1}^n \Ct(\X_i)^{\geo}/l^m } \ar@{^{(}->}[r]^-{\oplus \rho_{\X_i}}& \displaystyle{\bigoplus_{i=1}^n \pitab(\X_i)^{\geo}/l^m},
}
}
\]
where the right vertical map is given by the projection $\X \to \X_i$.
Since the bottom horizontal map is injective in the above diagram, 
so is $\rho_{\X}$ and the assertion follows from this. 
\end{proof}

\begin{proof}[Proof of Thm.~\ref{thm:main}] 
Using Lem.\ \ref{lem:key}, 
the same proof of Thm.\ 4.6 in \cite{Hir17} works well. 
We give a sketch of the proof. 
Let 
$\Np$ be the set of $m\in\Z_{\ge 1}$ which is prime to $p$.
For an abelian  group $G$, we denote by 
\[
G_{\L} := \plim_{m \in \Np}G/m.
\]
From Lem.\ \ref{lem:key}, 
the map 
$\rho_{X,\L}: \CX_{\L} \inj \piabX_{\L}$ induced from $\rhoX$  
is injective. 
Consider the following diagram:
	\begin{equation}
	\label{eq:Kerdiv}
	\vcenter{
	\xymatrix{
	            &                    & 0\ar[d]                 & 0\ar[d]\\  
		0\ar[r] & \Ker(\rhoX)\ar@{=}[d] \ar[r]& \Ker(\psi)\ar[d] \ar[r] & \Ker(\phi)\ar[d] \\
		0\ar[r] & \Ker(\rhoX)\ar[r]&  \CX \ar[d]^{\psi}\ar[r]^-{\rhoX} & \piabX\ar[d]^{\phi} \\
		& 0\ar[r] & \CX_{\L} \ar[r]^-{\rho_{X,\L} }& \piabX_{\L}, 
	}}
	\end{equation}
where 
$\psi: \CX \to \CX_{\L}$ and $\phi:\piabX \to \piabX_{\L}$ 
are natural maps. 
 
For any prime $l\neq p$, 
$\Ker(\psi)$ is $l$-divisible 
by using Lem.\ 7.7 in \cite{JS03} (\Cf \cite{Hir17}, Claim in the proof of Thm.~4.6). 
On the other hand, 
we have a short exact sequence
\[
\bigoplus_{C\in\CuX} \bigoplus_{x\in\Cinf}I^1_{k(C)_x} \to \Ker(\phi) \to \Ker(\phi^t) \to 0,
\]
where $\phi^t: \pitabX\to \pitabX_{\L}$ is the natural map. 
The wild inertia subgroup $I_{\kCx}^1$ is pro-$p$. 
From the finitely generatedness of $\pitabX$, 
$\Ker(\phi^t)$ is $l$-torsion free 
and so is $\Ker(\phi)$. 
From the top exact sequence in \eqref{eq:Kerdiv}, 
$\Ker(\rhoX)$ is $l$-divisible.
\end{proof}


\begin{thebibliography}{10}

\bibitem{BT73}
H.~Bass and J.~Tate, \emph{The {M}ilnor ring of a global field}, Algebraic
  {$K$}-theory, {II}: ``{C}lassical'' algebraic {$K$}-theory and connections
  with arithmetic ({P}roc. {C}onf., {S}eattle, {W}ash., {B}attelle {M}emorial
  {I}nst., 1972), Springer, Berlin, 1973, pp.~349--446. Lecture Notes in Math.,
  Vol. 342.

\bibitem{Bloch81}
S.~Bloch, \emph{Algebraic {$K$}-theory and classfield theory for arithmetic
  surfaces}, Ann. of Math. (2) \textbf{114} (1981), no.~2, 229--265.

\bibitem{FV}
I.~B. Fesenko and S.~V. Vostokov, \emph{Local fields and their extensions},
  second ed., Translations of Mathematical Monographs, vol. 121, American
  Mathematical Society, Providence, RI, 2002, With a foreword by I. R.
  Shafarevich.

\bibitem{SGA1}
A.~Grothendieck, \emph{{R}ev\^etements \'etales et groupe fondamental},
  Springer-Verlag, Berlin, 1971, S\'eminaire de G\'eom\'etrie Alg\'ebrique du
  Bois Marie 1960--1961 (SGA 1), Dirig\'e par Alexandre Grothendieck.
  Augment\'e de deux expos\'es de M. Raynaud, Lecture Notes in Mathematics,
  Vol. 224.

\bibitem{Hir10}
T.~Hiranouchi, \emph{Class field theory for open curves over {$p$}-adic
  fields}, Math. Z. \textbf{266} (2010), no.~1, 107--113.

\bibitem{Hir17}
\bysame, \emph{Class field theory for open curves over local fields}, to appear
  in J. Th\'eor.\ Nombres Bordeaux (2017).

\bibitem{JS03}
U.~Jannsen and S.~Saito, \emph{Kato homology of arithmetic schemes and higher
  class field theory over local fields}, Doc. Math. (2003), Extra Vol.,
  479--538 (electronic), Kazuya Kato's fiftieth birthday.

\bibitem{KY13}
B.~Kahn and T.~Yamazaki, \emph{Voevodsky's motives and {W}eil reciprocity},
  Duke Math. J. \textbf{162} (2013), no.~14, 2751--2796.

\bibitem{Kat87}
K.~Kato, \emph{Swan conductors with differential values}, Galois
  representations and arithmetic algebraic geometry (Kyoto, 1985/Tokyo, 1986),
  Adv. Stud. Pure Math., vol.~12, North-Holland, Amsterdam, 1987, pp.~315--342.

\bibitem{LiuAG}
Q.~Liu, \emph{Algebraic geometry and arithmetic curves}, Oxford Graduate Texts
  in Mathematics, vol.~6, Oxford University Press, Oxford, 2002, Translated
  from the French by Reinie Ern{{\'e}}, Oxford Science Publications.

\bibitem{Mil70}
J.~Milnor, \emph{Algebraic {$K$}-theory and quadratic forms}, Invent. Math.
  \textbf{9} (1969/1970), 318--344.

\bibitem{RS00}
W.~Raskind and M.~Spiess, \emph{Milnor {$K$}-groups and zero-cycles on products
  of curves over {$p$}-adic fields}, Compositio Math. \textbf{121} (2000),
  1--33.

\bibitem{Sai85}
S.~Saito, \emph{Class field theory for curves over local fields}, J. Number
  Theory \textbf{21} (1985), no.~1, 44--80.

\bibitem{Som90}
M.~Somekawa, \emph{On {M}ilnor {$K$}-groups attached to semi-abelian
  varieties}, $K$-Theory \textbf{4} (1990), no.~2, 105--119.

\bibitem{V}
V.~Voevodsky, \emph{Triangulated categories of motives over a field}, Cycles,
  transfers, and motivic homology theories, Ann. of Math. Stud., vol. 143,
  Princeton Univ. Press, Princeton, NJ, 2000, pp.~188--238.

\bibitem{Yam09}
T.~Yamazaki, \emph{Class field theory for a product of curves over a local
  field}, Math. Z. \textbf{261} (2009), no.~1, 109--121.

\bibitem{Yam13}
\bysame, \emph{The {B}rauer-{M}anin pairing, class field theory, and motivic
  homology}, Nagoya Math. J. \textbf{210} (2013), 29--58.

\end{thebibliography}
\def\cprime{$'$}
\providecommand{\bysame}{\leavevmode\hbox to3em{\hrulefill}\thinspace}
\providecommand{\MR}{\relax\ifhmode\unskip\space\fi MR }
\providecommand{\MRhref}[2]{%
  \href{http://www.ams.org/mathscinet-getitem?mr=#1}{#2}
}
\providecommand{\href}[2]{#2}

\bigskip\noindent
Toshiro Hiranouchi \\
Department of Basic Sciences, 
Graduate School of Engineering, 
Kyushu Institute of Technology\\
1-1 Sensui-cho, Tobata-ku, Kitakyushu-shi, 
Fukuoka, 804-8550, JAPAN\\
\texttt{hira@mns.kyutech.ac.jp}

\end{document}